\documentclass{article}
  
  \usepackage[utf8]{inputenc}
  \usepackage[T1]{fontenc}
  \usepackage{geometry}
  \usepackage[english]{babel}

\usepackage{stmaryrd}
\usepackage{amsfonts, amssymb, amsmath, amsbsy, upref, nicefrac}
\usepackage{ifpdf}
\usepackage{hyperref}
\usepackage{enumitem}
\usepackage[german,intoc]{nomencl}
\usepackage{xspace}
\makenomenclature

\usepackage{pst-all}
\usepackage{color}
\usepackage{epsfig}

\newtheorem{definition}{Definition}
\newtheorem{proof}{Proof}

\newtheorem{theorem}{Theorem}
\newtheorem{exercise}{Exercise}

\newtheorem{example}{Example}

\newtheorem{remark}{Remark}

\DeclareMathOperator{\id}{\mathrm{id}}

\newcommand{\e}{\mathrm{e}}
\newcommand{\ii}{\mathrm{i}}

\newcommand{\Iff}{\mbox{if and only if,}\xspace}

\newcommand{\ie}{\mbox{i.\,e.}\xspace}

\newcommand{\resp}{\mbox{respectively}\xspace}

\newcommand{\wlg}{\mbox{without loss of generality}\xspace}

\renewcommand{\d}{\mathrm{d}}

\newcommand{\dx}{\,\mathrm{d}x}

\newcommand{\dt}{\,\mathrm{d}t}

\newcommand{\N}{\mathbb{N}}
\newcommand{\Z}{\mathbb{Z}}

\newcommand{\R}{\mathbb{R}}

\newcommand{\vp}{\varphi}

  \title{More convex functions by Artin`s method}
  \author{Martin Himmel}
  \date{\today}
  
\begin{document}
  \maketitle
\begin{abstract}
We use Artin`s paper on the Gamma function to find more $\log$ convex functions 
that interpolate a sequence of natural numbers given by a recursion equation.
\end{abstract}

\section{Introduction}
Let us recall some ideas and results of from Artin`s famous paper on the Gamma function.
Therefore, let $a$ and $b$ be real numbers with $a<b$ and $f: (a,b) \to \R$ a function.
For any $x_1, x_2 \in (a,b)$, $x_1 \neq x_2$, we define the difference quotient 
\begin{equation}
\vp(x_1, x_2):=\frac{f(x_1)-f(x_2)}{x_1-x_2}=\vp(x_2,x_1).
\end{equation}
Then, for pairwise different $x_1, x_2, x_3 \in (a,b)$, the iterated difference quotient is defined by
\begin{equation}
\Phi(x_1, x_2, x_3):=\frac{\vp(x_1, x_3)-\vp(x_2, x_3)}{x_1-x_3}
\end{equation}

\begin{exercise}
Show that $\Phi(x_1, x_2, x_3)$ does not change sign under permutation of the arguments $x_1, x_2$ and $x_3$.
\end{exercise}

\begin{definition}[Convexity]
The function $f: (a,b) \to \R$ is called convex
if, for any fixed $x_3 \in (a,b)$, the difference quotient
$\vp(x_1, x_3)$ is a monotone increasing function of $x_1$,
\ie whenever $x_1, x_2 \in (a,b)$ with $x_1<x_2$ we have
\begin{equation}
\vp(x_1, x_3)\leq \vp(x_2, x_3).
\end{equation} 
\end{definition}
\begin{exercise}
Show that $f: (a,b) \to \R$ is convex, if and only if,
the iterated difference quotient $\Phi$ satisfies $\Phi(x_1, x_2, x_3) \geq 0$.
\end{exercise}

\begin{exercise}
Show that 
the sum $f+g$ of convex functions $f: (a,b) \to \R$ and $g: (a,b) \to \R$ is convex and 
the limit $\lim{f_n}$  of a sequence of convex functions $f_n: (a,b) \to \R$ is convex.
\end{exercise}
Much is known about convex functions \cite{SIMONConvexity}.
The following results are taken from Artin`s paper on the Gamma function \cite{ArtinGamma}.
\begin{theorem}[Rolle]
Let $f: (a,b) \to \R$ be continuous whose one-sided derivatives $f(x+0)$ and $f(x-0)$ exists for $x\in (a,b)$.
Moreover, assume that $f(a)=f(b)$.
Then, there is a $\xi \in (a,b)$ such that
\begin{equation}
f^\prime(\xi+0) \geq 0 \text{ and } f^\prime(\xi-0)\leq 0
\end{equation}
\end{theorem}

\begin{theorem}[Mean Value]
Let $f: [a,b] \to \R$ be continuous with one-sided derivatives $f(x+0)$ and $f(x-0)$ exists for any $x\in (a,b)$.
Then there exists  $\xi \in (a,b)$ such that
\begin{equation}
\frac{f(b)-f(a)}{b-a} \in \left(f^\prime(\xi-0), f^\prime(\xi+0)\right)
\end{equation}
\end{theorem}

\begin{theorem}[Characterization of Convexity]
$f: (a,b) \to \R$ is a convex function if, and only if,
$f$ has monotonically increasing one-sided derivatives.
If, in addition, $f$ is twice differentiable,
convexity of $f$ is equivalent to $f^{\prime \prime} \geq 0$ for $x\in (a,b)$.
\end{theorem}

\begin{definition}[Weak Convexity]
$f:(a,b)\to \R$ is called weakly convex if 
\begin{equation}
f(x_1+x_2) \leq \frac{1}{2}(f(x_1)+f(x_2))
\label{eq:weakConvexity}
\end{equation}
holds for all $x_1, x_2 \in (a,b)$.
\end{definition}

Any weakly convex function is convex
and the converse holds for continuous functions.
\begin{theorem}[Weakly Convex plus Continuous implies Convex]
$f: (a,b) \to \R$ is a convex function if, and only if,
$f$ is weakly convex and continuous.
\end{theorem}

\begin{exercise}
Let $f(x):=-\log(x)$ defined on the positive real numbers $\R_+$.
Show that $f$ is convex.
How is inequality \eqref{eq:weakConvexity} called in this situation? 
\end{exercise}

\begin{definition}[Log-Convexity]
$f:(a,b)\to \R$ is called log-convex%
\footnote{In literature you will also find logarithmically convex or superconvex for what we call here log-convex.}
(weakly log-convex) 
if $f$ is positive and $\log{f}$ is convex (weakly log-convex).
\end{definition}
The positivity assumption on $f$ for log-convexity is, of course,
a formal prerequisite because otherwise the logarithm of $f$ cannot be formed. 
Any logarithmically convex function is convex 
since it is the composite of the increasing convex function $\exp$ and the function $\log{f}$,
but the the converse generally does not hold.
\begin{exercise}
Verify that the function $x \mapsto x^2$ defined on $\R$ is convex, but not log-convex.
On the other hand: given a convex function $g: \R \to \R$, that is not log-convex. 
Can you find some fixed other function $l:\R\to\R$ such that
$l$ composed with $g$, which we call 
\begin{equation*}
h(x):=l(g(x)),
\end{equation*}
is log-convex?\\
\textbf{Hint:} Depending on your background you may call $l$ (if you pick the best/simplest $l$)
the most or the the second most important function in mathematics. 
\end{exercise} 
As we have seen, convex functions form a vector space and so do log-convex functions,
but since the product of log-convex functions is again log-convex, they even form an algebra.

\begin{theorem}
Let $f,g: (a,b) \to \R$ be two log-convex (weakly log convex) functions.
Then their sum $f+g$ 
and their product $f\cdot g$ 
is log-convex (weakly log convex).
The same holds for sequences of log-convex functions, 
if the limit function is positive.
\end{theorem}
\subsection{Log-convex Integrals}
The previous results can be combined to the statement on log-convexity of integrals.
Suppose $f:(a,b)\times I \to \R$, $I$ some interval of $\R$, is a continuous function 
of the two variables $x$ and $t$.
Furthermore, for any fixed value of $t$, suppose that $f(t, \cdot)$ is log-convex, 
twice differentiable function of $x$.
For any fixed integer $n$ we can build the function
\begin{equation}
F_n(x)=h \sum_{k=0}^{n-1}{f(a+k h,x)}
\end{equation}
with $h=\frac{b-a}{n}$.
Being the sum of log-convex functions,
$F_n$ is also log-convex. As $n$ approaches infinity,
$F_n$ converges to the integral
\begin{equation}
\int_a^b{f(t,x)\dt},
\end{equation}
which hence is also log-convex.
If $b=\infty$, the result also holds supposed the improper integral converges.
Artin is mainly interested in integral representations of Euler`s Gamma function
and therefore considers integrals of the form
\begin{equation}
\int_a^b{\vp(t)t^{x-1}\dt}
\label{eq:ArtinsLogConvexFunctions}
\end{equation}
with $\vp: (a,b) \to \R$ being a positive and continuous function for $t \in (a,b)$.
Verify that $\frac{\d^2}{\dx^2}{(\vp(t)t^{x-1})}=0$.

Sufficiently smooth log-convex functions are characterized in the next
\begin{theorem}[Characterization of Log-Convexity]
\label{theo:D2LogConvexity}
Let $f:(a,b)\to \R$ be twice differentiable and without zeros.
Then $f$ is log-convex if, and only if,
\begin{equation}
q(f):=
\det\begin{pmatrix}
f & f^\prime\\
f^\prime & f^{\prime \prime}	
\end{pmatrix}=
f f^{\prime \prime}-(f^{\prime})^2 \geq 0
\label{eq:smoothFunctionsLogConvexity}
\end{equation}
for all $x\in(a,b)$
\end{theorem}
\begin{proof}
Let $f:\R\to\R$ be twice differentiable and log-convex.
Thus $f^{\prime \prime}$ exists for all $x\in\R$ and the second derivative of $\log{f}$ is non-negative
\begin{eqnarray*}
\left(\log{f}\right)^{\prime \prime}&=&\left(\frac{f^\prime}{f}\right)^\prime \\
&=&
\frac{f^{\prime \prime}}{f}-\frac{(f^{\prime})^2}{f^2}=\frac{f f^{\prime \prime}-(f^{\prime})^2}{f^2}\\
&=&
\frac{1}{f^2}
\det{
\begin{pmatrix}
f & f^\prime\\
f^\prime & f^{\prime \prime}
\end{pmatrix}} \geq 0
\end{eqnarray*}
and hence
\begin{equation*}
f \geq \frac{(f^\prime)^2}{f^{\prime \prime}}.
\end{equation*}
\end{proof}

\begin{theorem}
Let $\vp:(a,b)\to \R$ be positive continuous function.
Then
\begin{equation}
\int_a^b{\vp(t)t^{x-1}\dt}
\end{equation}
is a log-convex function of $x$ defined where the integral converges.
\end{theorem}

\begin{theorem}
Let $f:(a,b)\to \R$ be a log-convex function and $c\in \R$, $c\neq 0$.
Then the translated $f_t:(a,b) \to \R$ and the scaled $f_{s}:(a,b) \to \R$ version of $f$ defined by
$f_{t}(x):=f(x+c)$
and $f_{s}(x):=f(cx)$, \resp,
are log-convex.
\end{theorem}
With more sophisticated words one can express the latter theorem as:
the space of log-convex functions 
is a common invariant subspace 
of the uniform shift and the uniform translation operator.

\section{Euler`s Gamma Function}
By the famous Bohr-Mollerup theorem Euler`s Gamma function can be characterized as the unique solution 
to the following interpolation problem.
\begin{theorem}[Bohr Mollerup]\label{thm:BohrMollerup}
Any function $f:\R\to\R$ that satisfies
\begin{enumerate}
	\item 
	$f(1)=1$,
	\item
	$f(x+1)=x f(x)$ for any  $x>0$,
	\item
	$f$ is log-convex,
\end{enumerate}
equals the Gamma function 
\begin{equation}
\Gamma(x):=\int_{0}^\infty{{\e^{-t} t^{x-1}}\dt}.
\label{eq:GammaFunction}
\end{equation}
\end{theorem}
Because of the functional Equation in Theorem \ref{thm:BohrMollerup} (second condition) the Gamma function is well-known
as an analytic continuation (from $\N$ to $\R$) of the sequence of factorials $n\mapsto n!=n \cdot (n-1)!$ with $1!:=1$.
\begin{exercise}
Verify by partial integration that
$\Gamma(n)=(n-1)!$
for any natural number $n\geq 2$.
\end{exercise}
Surprising is that log-convexity is the property that characterizes the Gamma function uniquely up to some normalization 
(first condition in Theorem \ref{thm:BohrMollerup}) as the only function that agrees at the natural number $n$ with $(n-1)!$.
\begin{exercise}
Verify that the Gamma function as defined in \eqref{eq:GammaFunction} satisfies the three properties from theorem \ref{thm:BohrMollerup}. 
\end{exercise}
\section{Artin Type Functional Equations}
We take Theorem \ref{thm:BohrMollerup} as a starting point
to constructively solve  multiplicative Functional Equation of the form
\begin{equation}
f(x+1)=g(x) f(x)
\label{eq:FG}
\end{equation}
with $g: \R \to \R$ some continuous function.
Many interesting functions satisfy a Artin functional equation \eqref{eq:FG} for some $g:\R\to\R$.
For instance, if $g(x)=1$ for all real numbers $x$, solutions $f$ to \eqref{eq:FG} are $1$-periodic functions.
If $g$ is the identity function on $\R$, 
then $f$ interpolates the factorials since $f(x+1)=x f(x)$.
Under additional assumptions (for instance, if $f$ is log-convex and $f(1)=1$, see Bohr Mollerup theorem \ref{thm:BohrMollerup})
we have $f=\Gamma$. 

\begin{definition}
Let $g:\R\to\R$ any function and $a>0$ and $f:(0,a)\to\R$ a solution to the functional equation \eqref{eq:FG}.
Then we call $f$ Artin function with representer $g$.
A log-convex Artin function is called Bohr Mollerup function or function of Bohr Mollerup type.
\end{definition}

\begin{exercise}
Find necessary and sufficient conditions on the representer $g$ of an Artin function $f$ such that
$f$ is of Bohr Mollerup type.
\end{exercise}
Of course, the notion of Artin functions is merely a tautolygy 
because every function $f:\R\to\R$ can be thought as an Artin function with representer $g(x):=\frac{f(x+1)}{f(x)}$.
We want to explore
whether there is some kind of analogue to the Bohr Mollerup theorem for functional equations of the form \eqref{eq:FG} 
when $g$ is not the identity function.
Therefore, we have to investigate when solutions to \eqref{eq:FG} are log-convex and satisfy some 
interpolation property $f(n)=a_n$ for some positive%
\footnote{The sequence $(a_n)_{n\in\N}$ has to be positive because otherwise $\log{f}$ cannot be formed
and there is no hope for log-convexity of $f$.}
sequence of real numbers $(a_n)_{n\in\N}$.
In the case of the Gamma function we have $g(x)=x$ on $\R$ and $f(n)=a_n:=(n-1)!$.

\begin{theorem}[Bohr Mollerup Type Representation]
Let $g: \R \to \R$ be a positive continuous function and $f:\R \to \R$ a solution to functional equation \eqref{eq:FG},
which satisfies
\begin{eqnarray}
f(1)=g(\infty):=1\label{eq:bmNormalization}\\ 
f(n)=\prod_{k=1}^{n-1}{g(k)} \label{eq:bmInterpolation} \\
f \text{ is log-convex.} \label{eq:bmLogConvexity}
\end{eqnarray}
Then $f:\R \to \R$ with 
\begin{equation}
f(x)= \lim_{n\to \infty}{{g^x{(n)} \prod_{k=0}^n{\frac{g(k)}{g(x+k)}}}}
\label{eq:BohrMollerupTypeRepresentation}
\end{equation}
is the only solution to the Bohr Mollerup functional equation \eqref{eq:FG}.
\end{theorem}

\begin{proof}

\begin{itemize}
	\item[\textbf{Step 1}] 
	Find a product formula for the solution $f:(0,1]\to \R$ to the Bohr Mollerup functional equation \eqref{eq:FG} by
	using the log-convexity \eqref{eq:bmLogConvexity} and the interpolation property \eqref{eq:bmInterpolation}.
	\item[\textbf{Step 2}]
	Extend this solution to $\R$ by iterated application of \eqref{eq:FG}.
	\item[\textbf{Step 3}]
	Show that the product representation obtained in step $1$ holds on $\R$. 
	\end{itemize}
\textbf{Step 2:}
Suppose a log-convex solution $f:(0,1]\to \R$ to \eqref{eq:FG} is known. 
Then $f$ can be extended to the interval $(0,n]$ by iterated application of the functional equation \eqref{eq:FG}:
\begin{eqnarray*}
f(x+n)&=&g(x+n-1) f(x+n-1)=g(x+n-1) g(x+n-2) f(x+n-2)\\
&=& g(x+n-1) g(x+n-2) g(x+n-3) f(x+n-3)=\ldots\\
&=& g(x+n-1) g(x+n-2) g(x+n-3) \cdot\ldots\cdot g(x+2)g(x+1)g(x)f(x)\\
&=& \prod_{k=0}^{n-1}{g(x+k)} f(x)
\end{eqnarray*}
Now we want to extend $f$ to include negative real numbers.
Therefore, we solve the iterated functional equation $f(x+n)=\prod_{k=0}^{n-1}{g(x+k)} f(x)$
for $f(x)$ and take the expression obtained as a definition of $f$ for negative real numbers:
if $x$ lies in the interval $(-n,-n+1)$ for some $n\in\N$, we define the value of $f$ to be
\begin{equation}
f(x)=f(x+n)\frac{1}{\prod_{k=0}^{n-1}{g(x+k)}}
\label{eq:gGammaNegative}
\end{equation}
\textbf{Step 1:}
Now, let`s use the log-convexity of $f$ to find the exact value $f(x)$ for $0<x\leq 1$.
Let $n\geq2$. Then
\begin{equation}
\frac{\log{f(n-1)}-\log{f(n)}}{(n-1)-n} \leq \frac{\log{f(x+n)-\log{f(n)}}}{(x+n)-n}\leq \frac{\log{f(1+n)}-\log{f(n)}}{(n+1)-n}
\label{eq:monotoneIncreaseLogConvex}
\end{equation}
expresses the monotone increase of the difference quotients of $\log{f}$.
If we use the interpolation property $f(n)=\prod_{k=1}^{n-1}{g(k)}$,
equation \eqref{eq:monotoneIncreaseLogConvex} reads
\begin{equation}
\log{g(n-1)}\leq \frac{\log{f(x+n)}-\log{f(n)}}{x} \leq \log{g(n)}
\end{equation}
and hence
\begin{equation}
g^x{(n-1)} \leq \frac{f(x+n)}{f(n)} \leq g^x{(n)}.
\end{equation}
Multiplying by $f(n)=\prod_{k=1}^{n-1}{g(k)}$  and using the iterated functional equation $f(x+n)=f(x) \prod_{k=0}^{n-1}{g(x+k)}$
we obtain 
\begin{equation}
g^x{(n-1)} \leq f(x) \frac{\prod_{k=0}^{n-1}{g(x+k)}}{\prod_{k=1}^{n-1}g(k)} \leq g^x{(n)}.
\end{equation}
Setting $p_n(x):=\prod_{k=0}^{n-1}{\frac{g(k)}{g(x+k)}}$ with $g(0):=1$ the latter equation reads
\begin{equation}
p_n(x) \cdot g^x{(n-1)} \leq f(x) \leq p_n(x) \cdot g^x{(n)}.
\end{equation}
Since this holds for all $n \geq 2$,
we can replace $n$ by $n+1$ on the left side.
Thus
\begin{equation}
p_{n+1}(x) \cdot g^x{(n)} \leq f(x) \leq p_n(x) \cdot g^x{(n)}.
\label{eq:sandwitchGGamma}
\end{equation}
Now observe that $f$ is bounded from above and from below by the same function $p_{n+1}(x) \cdot g^x{(n)}$ 
up to the factor $\frac{g(x+n)}{g(n)}$
because $p_n(x)  = p_{n+1}(x)  \frac{g(x+n)}{g(n)}$.
Hence \eqref{eq:sandwitchGGamma} reads
\begin{equation}
p_{n+1}(x) \cdot g^x{(n)} \leq f(x) \leq p_{n+1}(x)  \cdot \frac{g(x+n)}{g(n)} \cdot g^x{(n)}.
\label{eq:eq:sandwitchGGamma2}
\end{equation}
Multiplying through the second part of \eqref{eq:eq:sandwitchGGamma2} by $\frac{g(n)}{g(x+n)}$ 
and combining it with the first part of \eqref{eq:eq:sandwitchGGamma2} gives
\begin{equation}
f(x) \frac{g(n)}{g(x+n)} \leq p_{n+1}(x)   \cdot g^x{(n)} \leq f(x).
\label{eq:eq:sandwitchGGamma3}
\end{equation}
Assuming
$\lim_{n \to \infty}{\frac{g(n)}{g(x+n)}}=1$ 
gives the desired product representation of $f$:
\begin{equation}
f(x) = \lim_{n\to \infty}{{g^x{(n)} \prod_{k=0}^n{\frac{g(k)}{g(x+k)}}}}
\label{eq:BohrMollerupProductRepresentaion}
\end{equation}
on $(0,1]$.\\
\textbf{Step 3:}
To see that representation \eqref{eq:BohrMollerupProductRepresentaion} holds even for all$ x \in \R$,
define 
\begin{equation*}
f_n(x):={{g^x{(n)} \prod_{k=0}^n{\frac{g(k)}{g(x+k)}}}}
\end{equation*}
to be the expression in \eqref{eq:BohrMollerupProductRepresentaion} under the limit sign.
Then we have
\begin{equation}
f(x+1)=g(x) f_n(x) \cdot \frac{g(n)}{g(x+n+1)}
\end{equation}
ans thus
\begin{equation*}
f_n(x)=f(x+1) \frac{g(x+n+1)}{g(x)g(n)}.
\end{equation*}
We see: if the limit in \eqref{eq:BohrMollerupProductRepresentaion} exists for $x$,
it also exists for $x+1$ and vice versa.
Hence the product representation \eqref{eq:BohrMollerupProductRepresentaion} is valid for all $x\in\R$. 
\end{proof}

Following the lines of Artin`s paper \cite{ArtinGamma}, we proceed by deriving expressions for $\log{f}$
with $f$ being a log-convex solution to the Bohr Mollerup Type functional equation \eqref{eq:FG}.
Assuming continuity of $f$ in \eqref{eq:BohrMollerupProductRepresentaion}, we obtain
\begin{equation}
\log{f(x)}=\lim_{n\to\infty}{\left(x\log{g(n)}+\sum_{k=0}^{n}{(\log{g(k)-\log{g(x+k)}})}\right)}
\label{eq:LogBohrMollerup}
\end{equation}
Now we would like to differentiate twice under the limit sign to obtain conditions on $g$ that guarantee the log-convexity of $f$.
If the convergence in \eqref{eq:LogBohrMollerup} is uniformly, we have
\begin{equation}
\left(\log{f(x)}\right)^{\prime \prime}={\sum_{k=0}^{\infty}{\left(\frac{(g^{\prime}(x+k))^2}{g^2(x+k)}
-\frac{g^{\prime \prime}(x+k)}{g(x+k)}
\right)}} \geq 0
\label{eq:LogConvexArtinFunction}
\end{equation}

\begin{remark}
For the Gamma function we have $g(x)=x$ and condition \eqref{eq:LogConvexArtinFunction} reads
\begin{equation}
\sum_{k=0}^{\infty}{\frac{1}{(x+k)^2}} \geq 0.
\end{equation}
\end{remark}

We can use \eqref{eq:LogConvexArtinFunction} to define an inner product.
\begin{definition}
Let $g:\R\to\R$ be twice differentiable.
Then
\begin{equation*}
a(f(x),g(x)):=
{\sum_{k=0}^{\infty}
{\det\begin{pmatrix}
(f^\prime (x+k))^2 & g^{\prime \prime}(x+k)\\
g^{-2}(x+k) & f^{-2}(x+k)   	
\end{pmatrix}}}
\end{equation*}
\end{definition}
Then log-convex solutions to the Artins functional equation \eqref{eq:FG} can be interpreted 
as functions that make the quadratic form 
$q(g):=a(g,g)$ positive.

\begin{example}
Setting $g=\id$ you obtain the product representation of Euler`s Gamma function.
Because $g(x)=0$ if, and only if, $x=0$, the function $g=\id$ does not satisfy the positivity assumption,
that we made. 
\end{example}

\begin{exercise}[Artin functions with Artinian derivative]
The derivative $f^\prime$ of an Artin function $f$ is called Artinian if $f^\prime$ again is an Artin function,
\ie from $f$ satisfying $f(x+1)=g_0(x)f(x)$ for some analytic function $g_0$ follows
that there is an analytic function $g_1$ such that $f^\prime(x+1)=g_1(x) f^\prime (x)$.
Show that the $n$-th derivative of an Artin function $f$ with representer $g_0$ is Artinian if  
\begin{equation}
g_n (x) = \frac{\left(g_{n-1}f^{(n-1)}\right)^\prime}{f^{(n)}}
\end{equation}
\end{exercise}

\begin{example}
Let us consider a class of Artin functions with representer
$g(x)=x^c$ with $c$ some complex number.
When are these Artin functions $f:=f_g$ satisfying $f(x+1)=g(x)f(x)$ of Bohr Mollerup type?
Condition \eqref{eq:LogConvexArtinFunction} reads
\begin{eqnarray*}
(\log{f})^{\prime \prime}
&=&
\sum_{k=0}^{\infty}{\frac{c(x+k)^{2(c-1)}}{(x+k)^{2c}} - \frac{c(c-1)(x+k)^{c-2}}{(x+k)^{c}}}\\
&=&
c(2-c)\sum_{k=0}^{\infty}{(x+k)^{-1}} \geq 0.
\end{eqnarray*}
Hence the Artin function $f$ is of Bohr Mollerup Type,
if $\Re{(c(2-c))}\geq 0$.
Consequently, $f$ is log-convex for all real numbers $c$.
\end{example}

\section{Making functions log-convex}
Assume $f:\R\to\R$ is a twice differentiable, but not log-convex function.
How can we modify $f$ such that the modified version of $f$ is log-convex?
\subsection{Inner multiplication problem}
Let $f:\R\to\R$ be a twice differentiable, but not log-convex function.
In the inner multiplication problem we seek to find a twice differentiable function $m:\R\to\R$ such that
$mf$ is log-convex. Any function such that $mf$ is log-convex is called inner multiplicator of $f$.
According to theorem \ref{theo:D2LogConvexity} $mf$ is log-convex 
\Iff second derivative of $\log{(mf)}$ is non-negative. 
This is characterized in the next
\begin{theorem}
Let $f:\R\to\R$ be a differentiable, but not log-convex function.
Then, there is an inner multiplicator $m: \R \to \R$ of $f$ \Iff
\begin{equation}
mf \cdot (mf)^{\prime \prime} \geq ({(mf)}^{\prime})^2.
\label{eq:innermultiplierLogconvex}
\end{equation}
\end{theorem}

\subsection{Outer multiplication problem}
Let $f:\R\to\R$ be a twice differentiable, but not log-convex function.
The outer multiplication problem consists of finding a twice differentiable function $m:\R\to\R$ such that
$m\log{f}=\log{(f^m)}$ is convex.
We use theorem \ref{theo:D2LogConvexity} to characterize convexity of $\log{(f^m)}$
and obtain 
\begin{theorem}
Assume $f:\R\to\R$ is twice differentiable, but not log-convex.
Then an outer multiplier of $f$ satisfies
\begin{equation}
(m\log{f})^{\prime \prime}=
m^{\prime \prime} \log{f}+ 2 m^\prime \frac{f^\prime}{f}+
m \underbrace{\left(\frac{f^{\prime \prime}}{f}-{\left(\frac{f^{\prime}}{f}\right)}^2\right)}_{(\log{f})^{\prime \prime}\ngeq 0}\geq 0.
\end{equation}
\end{theorem}

\begin{theorem}[Curvature of Artin functions]
Let $f$ be a twice differentiable Artin function with representer $g$:
\begin{equation*}
f(x+1)=g(x) f(x).
\end{equation*}
Then the curvature of $f$ is given by
\begin{equation*}
\kappa_f=\frac{g^{\prime \prime}f+2g^\prime f^\prime + g f^{\prime\prime}}{\left(1+(g^\prime f+g f^\prime)^2\right)^{\frac32}}.
\end{equation*}
\end{theorem}

\section{Fibonacci function}
The goal of this section is to find a real-valued log-convex function that interpolates the Fibonacci numbers,
which are given by the following linear second order recursion equation
\begin{equation}
a_n=a_{n-1}+a_{n-2}
\label{eq:FibonacciRecursion}
\end{equation}
with initial values $a_{0}=a_{1}=1$.
The first Fibonacci numbers are $1,1,2,3,5,7,11,18$.
For any linear recursion of depth two we can find a closed formula by making the ansatz $a_{n}=\lambda^n$.
Here this ansatz leads to the quadratic equation $\lambda^2-\lambda+1=0$
with solution $\lambda_{1,2}=\frac{1}{2}(1\pm \sqrt{5})$.
Hence the $n$-th Fibonacci number can be calculated directly by
\begin{equation}
a_n=\frac{\vp^n-(-1)^n\vp^{-n}}{\sqrt{5}}
\label{eq:closedFormFibonacci}
\end{equation}
with $\vp=\frac{1+\sqrt{5}}{2}$ the golden ratio.
If we replace $n$ in \eqref{eq:closedFormFibonacci} by some real number $x$,
we get an extension of the Fibonacci numbers 
\begin{equation}
F(x)=\frac{\vp^x-(-1)^x\vp^{-x}}{\sqrt{5}}.
\label{eq:complexFibonacciNumber}
\end{equation}
Since $(-1)^x=\e^{\ii \pi x}=\cos{(\pi x)}+ \ii \sin{(\pi x)}$, 
this extension is not always real-valued,
more precisely $F(x)$ is a real number \Iff $x$ is an integer.
We are interested in constructing a real-valued extension of the Fibonacci numbers.
Therefore, we consider the real part of $F(x)$
\begin{equation}
f(x):=\Re{F(x)}=\frac{\vp^x-\cos{(\pi x)}\vp^{-x}}{\sqrt{5}}.
\end{equation}
Let us take a closer look on the behavior of $f(x)$ at the integers.
Since $\cos{\pi n}=(-1)^n$ for $n \in \Z$, we have 
\begin{equation}
f(n)=
\begin{cases}
\frac{\sqrt{5}}{2} \sinh{(n \ln{\vp})}, \quad n \text{ even}\\
\frac{\sqrt{5}}{2} \cosh{(n \ln{\vp})}, \quad n \text{ odd}
\end{cases}
\label{eq:FibonacciExtensionIntegers}
\end{equation}
Since the Fibonacci numbers form an increasing sequence of natural numbers, 
the real part of its canonical interpolation $f$ is an Artin function 
with representer
\begin{equation}
g(x):=\frac{f(x+1)}{f(x)}=
\frac{\vp^{x+1}-\cos{(\pi x)} \vp^{-(x+1)}}{\vp^x-\cos{(\pi x)} \vp^{-x}}.
\label{eq:representerFibonacci}
\end{equation}
Is $f:\R\to\R$, the real part of the canonical extension of the Fibonacci numbers to $\R$, log-convex?
No. Due to the oscillating term $\cos{(\pi x)}$ the second derivative of $f$ changes sign four times in the intervall $[0,4]$.

\bibliography{bibliography/myrefs}
\bibliographystyle{plain}
 
  \end{document}